\documentclass[12pt]{amsart}

\usepackage{amssymb,amsthm}

\newtheorem{theorem}{Theorem}[section]
\newtheorem{corollary}[theorem]{Corollary}

\newtheorem{lemma}[theorem]{Lemma}

\def\irr#1{{\rm  Irr}(#1)}
\def\irri#1#2{{\rm Irr}_{#1} (#2)}

\def\ker#1{{\rm ker} (#1)}

\def\phi{{\varphi}}

\newcommand{\nli}[2]{{\rm nl}_{#1} (#2)}

\newcommand{\QQ}{{\mathbb Q}}
\newcommand{\RR}{{\mathbb R}}
\newcommand{\CC}{{\mathbb C}}
\newcommand{\acdd}[2]{{\rm acd}_{#1} (#2)}
\newcommand{\acd}[1]{\acdd {}{#1}}

\begin{document}

\title[Average character Degrees]{Variations on average character degrees and $p$-nilpotence}
\author[Mark L. Lewis]{Mark L. Lewis}

\address{Department of Mathematical Sciences, Kent State University, Kent, OH 44242}
\email{lewis@math.kent.edu}

\subjclass[2010]{ 20C15.}
\keywords{character degrees, $p$-nilpotence, averages }

\begin{abstract}
We prove that if $p$ is an odd prime, $G$ is a solvable group, and the average value of the irreducible characters of $G$ whose degrees are not divisible by $p$ is strictly less than $2(p+1)/(p+3)$, then $G$ is $p$-nilpotent.  We show that there are examples that are not $p$-nilpotent where this bound is met for every prime $p$.  We then prove a number of variations of this result.
\end{abstract}

\maketitle

\section{Introduction}

Throughout this paper, $G$ will be a finite group.  We write $\irr G$ for the set of irreducible characters of $G$.  The average value of the degrees of the irreducible characters can be written as
$$
\acd G = \frac {\sum_{\chi \in \irr G} \chi (1)}{|\irr G|}.
$$
So far as we can tell, $\acd G$ was first studied by Magaard and Tong-Viet in Theorem 1.4 of \cite{MaTV} where they showed that if $\acd G \le 2$, then $G$ is solvable.  Isaacs, Loukaki, and Moret\'o improved this in Theorem A in \cite{ILM} to show that if $\acd G \le 3$, then $G$ is solvable.  This question was finally settled by Moret\'o and Hung in Theorem A of \cite{MoNg} where they showed that if $\acd G < 16/5$, then $G$ is solvable.  This bound is best possible since $\acd {A_5} = 16/5$.  Given the success of using the average of the degrees of the characters to determine solvability, this quantity has been used to determine other properties.  Qian has studied $r$-solvability in terms of $\acd G$ in \cite{qian}.  Isaacs, Loukaki, and Moret\'o considered in \cite{ILM} the properties of supersolvability and nilpotence in terms of $\acd G$.  In particular, in Theorem C of \cite{ILM}, they proved that if $\acd G < 4/3$, then $G$ is nilpotent, and they showed that this bound is best possible.

Moret\'o and Hung suggested in \cite{MoNg} that it may be profitable to look at a variation of this average. In particular, they suggested looking at the average of the degree of the irreducible characters whose degrees are not divisible by some fixed prime $p$.  This idea was pushed further by Hung in \cite{nguyen}.  Let $p$ be a prime and let $\irri {p'}G$ be the set of irreducible characters whose degrees are not divisible by $p$, and write $\acdd {p'}G$ for the average value of the degrees of the characters in $\irri {p'}G$. In Theorem 1.2 of \cite{nguyen}, Hung proved that that if $\acdd {2'}G < 3$, $\acdd {3'}G < 3$, $\acdd {5'}G < 11/4$, or $\acdd {p'}G < 16/5$ when $p > 5$, then $G$ is solvable.  He also gave examples to see that each of these bounds are best possible.

Given that $\acd G$ can determine nilpotence, it was natural to see if $\acdd {p'}G$ determines $p$-nilpotence.  In fact, Hung was able to show in Theorem 1.1 of \cite{nguyen} that if $p = 2$ and $\acdd {2'}G < 3/2$ or if $p$ is odd and $\acdd {p'}G < 4/3$, then $G$ is $p$-nilpotent.  It is not difficult to see that these bounds are best possible for $p = 2$ and $p = 3$, but Hung \cite{nguyen} suggested that for odd primes $p$ that the bound should be $2(p+1)/(p+3)$.  In this paper, we prove that in fact this is the best bound.  In particular, we prove the following.

\begin{theorem} \label{first}
Let $G$ be a solvable group and let $p$ be an odd prime.  If $\acdd {p'}G < 2(p+1)/(p+3)$, then $G$ is $p$-nilpotent.
\end{theorem}

Note that if $D_{2p}$ is the dihedral group of order $2p$, then $\acdd {p'}{D_{2p}} = 2(p+1)/(p+3)$, so the bound in Theorem \ref{first} is best possible.  We also recall that Hung proved in Theorem 1.2 of \cite{nguyen} that if $\acdd {p'}G < 11/4$, then $G$ is solvable.  Since $2(p+1)/(p+3) < 2 < 11/4$, we can Theorem 1.2 \cite{nguyen} to remove the solvable hypothesis from Theorem \ref{first}.

\begin{corollary} \label{first a}
Let $G$ be a group and let $p$ be a prime.  If $\acdd {p'}G < 2(p+1)/(p+3)$, then $G$ is $p$-nilpotent.
\end{corollary}

In \cite{nguyen}, Hung showed that one could obtain similar results by further restricting the characters that are used to compute the average values.  In particular, if $k$ is a field, then $\irri kG$ is the set of irreducible characters of $G$ whose values lie in $k$ and $\irri {k,p'}G$ is the set of irreducible characters of $G$ with values in $k$ and whose degrees are not divisible by $p$.  We write $\acdd kG$ to denote the average value of the degrees of the characters in $\irri kG$ and $\acdd {k,p'}G$ for the average value of the degrees in $\irri {k,p'}G$.  In Theorem 1.3 (i) and Corollary 1.4 (i), (iii), and (iv) of \cite{nguyen}, Hung showed that if any one of $\acdd {\QQ}G$, $\acdd {\QQ,2'}G$, $\acdd {\RR}G$, or $\acdd {\RR,2'}G$ is strictly less than $3/2$, then $G$ has a normal $2$-complement.  He suggested that the bound for these inequalities should be $2$.  We are now able to prove that this bound holds.

\begin{theorem} \label{second}
Let $G$ be a solvable group.  If one of the following holds:
\begin{enumerate}
\item $\acdd {\QQ,2'}G < 2$,
\item $\acdd {\QQ}G < 2$,
\item $\acdd {\RR,2'}G < 2$,
\item $\acdd {\RR}G < 2$,
\end{enumerate}
then $G$ is $2$-nilpotent.
\end{theorem}

Notice that all of the irreducible characters of $S_4$ are rational valued, so $\acdd {\QQ}{S_4} = \acdd {\RR}{S_4} = \acd {S_4} = (1 + 1 + 2 + 3 + 3)/5 = 2$ and $\acdd {\QQ,2'}{S_4} = \acdd {\RR,2'}{S_4} = \acdd {2'}{S_4} = (1 + 1 + 3 + 3)/4 = 2$.  This implies that the bounds in Theorem \ref{second} are best possible.  Using Theorem 10.2 of \cite{nguyen}, we can remove the solvability hypothesis on Theorem \ref{second}.

When $p$ is an odd prime, we let $\QQ_p$ be the cyclotomic extension of $\QQ$ by a $p$th root unity.  In Theorem 1.3 (ii) and Corollary 1.4 (ii) of \cite{nguyen}, Hung showed that if either $\acdd {\QQ_p,p'}G$ or $\acdd {\QQ_p}G$ is strictly less than $4/3$, then $G$ has a normal $p$-complement.  Again, he suggested that $2(p+1)/(p+3)$ was the correct bound.  We prove that this bound holds.

\begin{theorem} \label{third}
Let $G$ be a a solvable group, and let $p$ be an odd prime.  If one of the following holds:
\begin{enumerate}
\item $\acdd {\QQ_p,p'}G < 2(p+1)/(p+3)$,
\item $\acdd {\QQ_p}G < 2(p+1)/(p+3)$,
\end{enumerate}
then $G$ is $p$-nilpotent.
\end{theorem}

We note that all the irreducible characters of $D_{2p}$ lie in $\QQ_p$ and have $p'$-degree; so $\acdd {\QQ_p}{D_{2p}} = \acdd {\QQ_p,p'}{D_{2p}} = \acdd {p'}{D_{2p}}$.  As we saw before, this implies that the bounds in Theorem \ref{third} are best possible.  Again, we can use Theorem 10.2 of \cite{nguyen} to remove the solvability hypothesis from Theroem \ref{third}.

If we restrict our attention to groups with odd order, we can further improve the bounds in Theorem \ref{third}.

\begin{theorem} \label{fourth}
Let $G$ be a group with odd order, and let $p$ be an odd prime.  If one of the following holds:
\begin{enumerate}
\item $p = 7$ and $\acdd {7'}G < 9/5$,
\item $p \ne 7$ and $\acdd {p'}G < 2$,
\item $\acdd {Q_p,p'}G < 2$,
\item $\acdd {Q_p}G < 2$,
\end{enumerate}
then $G$ is $p$-nilpotent.
\end{theorem}

Note that if $G$ the Frobenius group of order $21$, then $\acdd {7'}G = 9/5$, so the bound in Hypothesis (1) of Theorem \ref{fourth} is best possible.  We do not have examples where the other bounds in the hypotheses of Theorem \ref{fourth} are met, and we suspect that we do not have the best bound in those cases.

I would like to thank Nguyen Hung for bringing this problem to my attention and for providing me with a preprint of \cite{nguyen}.

\section{Odds and Ends}

We begin a lemma about averages that has to be well-known, but we want to explicitly prove it here.

\begin{lemma}\label{averages}
Let $A$ and $B$ be sets of real numbers (allowing repetitions) and let $C$ be the concatenation of $A$ and $B$.  If there is an integer $d$ so that $b \ge d$ for all $b \in B$ and ${\rm ave} (A) \ge d$, then ${\rm ave} (C) \ge d$.  In particular, if $b \ge {\rm ave} (A)$ for all $b \in B$, then ${\rm ave} (A) \le {\rm ave} (C)$.
\end{lemma}

\begin{proof}
We have $|C| = |A| + |B|$.  Observe that $\sum_{a \in A} a = {\rm ave} (A) |A|$.  Since every element $b \in B$ satisfies $b \ge d$, we have $\sum_{b \in B} b \ge d |B|$.  With this in mind, we calculate:
$$
{\rm ave} (C) = \frac {\sum_{a \in A} a + \sum_{b \in B} b}{|A|+ |B|} \ge \frac {d |A| + d |B|}{|A| + |B|} = d \frac {|A|+|B|}{|A|+|B|} = d,
$$
as desired.  For the second conclusion, take $d = {\rm ave} (A)$.
\end{proof}

We also need a quick result from calculus.

\begin{lemma} \label{calculus}
Let $p \ge 2$ be a fixed real number.  Then the function $f (x) = 2(p^x + 1)/(p^x + 3)$ is an increasing function on the interval $[1,\infty)$ and so its minimum value on that interval is $f (1) = 2(p+1)/(p+3)$.
\end{lemma}

\begin{proof}
This is an exercise in calculus.  Note that $f (x)$ is a continuous function.  The derivative of $f (x)$ is
$$
f'(x) = \frac {(p^x + 3)(2 \ln (p) p^x) - 2(p^x + 1) \ln (p) p^x}{(p^x + 3)^2}.
$$
Observe that $(p^x + 3)(2 \ln (p) p^x) - 2(p^x + 1) \ln (p) p^x = 4 \ln (p) p^x$ is clearly positive since $\ln p \ge \ln 2 > 0$.  This implies that $f'(x)$ is always positive on $[1,\infty)$, and so, $f (x)$ is an increasing function.
\end{proof}

Finally, we get to some group theory.  We need an easy result about the characters of $G$ when $G$ has a normal subgroup $K$ so that $K \cap G' = 1$.

\begin{lemma}\label{cent bij}
Let $K$ be a normal subgroup of $G$ so that $K \cap G' = 1$.  If $\phi \in \irr K$, then $\phi$ extends to $\tilde \phi \in \irr G$ and the map $f: \irr {G/K} \rightarrow \irr {G \mid \phi}$ by $f (\chi) = \chi \tilde\phi$ is a bijection such that $f (\chi) (1) = \chi (1)$ for all $\chi \in \irr {G/K}$.
\end{lemma}

\begin{proof}
Notice that $G'$ is contained in the kernel of $1_{G'} \times \phi \in \irr {G' \times K}$.  Since $G/G'$ is abelian, $1_{G'} \times \phi$ has an extension $\tilde\phi \in \irr {G/G'}$.  The bijection is easy to see, but one can refer to Gallagher's theorem (Corollary 6.17 of \cite{text}) if one wants a detailed proof.
\end{proof}

Now we consider characters whose values lie in a subfield of the complex numbers.  Let $k$ be a subfield of the complex numbers and we note that $k$ contains the rational numbers.  Let $G$ be a group.  Observe that if $\lambda$ is a linear character of $G$ with values in $k$, then $\lambda^{-1}$ will also be a linear character of $G$ with values in $k$.  If in addition $\delta$ is a linear character of $G$ with values in $k$, then $\lambda \delta$ will be a linear character with values in $k$.  Also, the principal character will have values in $k$.  Therefore, the set of linear characters with values in $k$ forms a subgroup of $\irr {G/G'}$.  Let $A^k (G)$ be the intersection of the kernels of the linear characters of $G$ with values in $k$, and observe that $\irr {G/A^k (G)}$ is the set of linear characters with values in $k$.  We see that $G' \le A^k (G) \le G$.

Since $\irri kG$ is the set of all irreducible characters with values in $k$, it is not difficult to see that $\irri kG \cap \irr {G/G'}$ will be the set of linear characters with values in $k$, and so, $\irri kG \cap \irr {G/G'} = \irr {G/A^k (G)}$.  Notice that if $p$ is a prime, it will still be the case that $\irri {k,p'}G \cap \irr {G/G'} = \irr {G/A^k (G)}$.

Observe that if $k = \CC$, then $A^k (G) = G'$ and if $k = \QQ$, then $A^k (G) = A^2 (G)$.  When $p$ is an odd prime and $k = \QQ_p$, it is not difficult to see that $A^{Q_p} (G) = A^2 (G) \cap A^q (G)$.

Note that if $H$ is a subgroup of $G$, then $H \cap A^k (G)$ will be a normal subgroup of $H$ and since $H/(H \cap A^k (G)) \cong HA^K(G)/A^k (G)$, we see that all of the characters in $\irr {H/(H \cap A^k (G))}$ are linear characters with values in $k$.  This implies that $\irr {H/(H \cap A^k (G))} \subseteq \irr {H/A^k (H)}$ and so, $A^k (H) \le H \cap A^k (G)$.

We apply this in the case where there is a minimal normal subgroup $K$ of $G$ so that $G' \cap K = 1$.

\begin{theorem} \label{acd cent k}
Suppose $k$ is a field, $G$ is a group, and $p$ is a prime.  Assume that $K$ is a minimal normal subgroup of $G$ such that $G' \cap K = 1$.  If $\acdd {k,p'}G \le 2$, then $\acdd {k,p'}{G/K} \le \acdd {k,p'}G$.
\end{theorem}

\begin{proof}
Notice that $[K,G] \le K \cap G' = 1$, so $K$ is central in $G$.  This implies that $|K|$ is a prime since it is minimal normal in $G$.  By Lemma \ref{cent bij}, we know that $\phi \in \irr K \setminus \{ 1_k \}$ has an extension to $G$.  If $A^k (G) \cap K = 1$, then $\phi \times 1_{A^k (G)}$ is a character of $K \times A^k (G)$ whose kernel is $A^k (G)$.  Since $G/A^k (G)$ is abelian, $\phi \times 1_{A^k (G)}$ will extend to a character in $\irr {G/A^k (G)}$.  Thus, $\phi$ will have an extension with values in $k$.  On the other hand, if $\phi$ has an extension $\tilde \phi$ with values in $k$, then $A^k (G)$ will lie in the kernel of $\tilde \phi$, and so, $A^k (G) \cap K$ will lie in the kernel of $\phi$, and since $\phi$ is faithful, we have $A^k (G) \cap K = 1$.  Notice that if $A^k (G) \cap K \ne 1$, then $K \le A^k (G)$ by the minimality of $K$.

Suppose first that $A^k (G) \cap K = 1$.  We have for each character $\phi \in \irr K$ there is an extension $\tilde\phi \in \irr {G/A^k (G)}$.  Applying Lemma \ref{cent bij}, we have for each $\phi \in \irr K$ that multiplication by $\tilde\phi$ yields a bijection from $\irr {G/K}$ to $\irr {G \mid \phi}$.   Notice that $\chi \in \irr {G/K}$ has $p'$-degree and values in $k$ if and only if $\chi \tilde\phi$ has $p'$-degree and values in $k$.  Since $K$ is central in $G$, we have $|K| = |\irr K|$. This implies that
$$
|\acdd {k,p'}G| = \sum_{\phi \in \irr K} |\acdd {k,p'}{G \mid \phi}| = |K| |\acdd {k,p'}{G/K}|,
$$
and
$$
\sum_{\chi \in \irri {k,p'}G} \chi (1) = \sum_{\phi \in \irr K} \sum_{\chi \in \irri {k,p'}{G \mid \phi}} \chi (1) = |K| \sum_{\chi \in \irri {k,p'}{G/K}} \chi (1).
$$
With these values, we compute that
$$
\acdd {k,p'}G =
\frac {|K| \sum\limits_{\chi \in \irri {k,p'}{G/K}} \chi (1)}{|K||\irri {k,p'}{G/K}|} = \frac {\sum\limits_{\chi \in \irri {k,p'}{G/K}} \chi (1)}{|\irri {k,p'}{G/K}|} = \acdd {k,p'}{G/K}.
$$

We now suppose that $K \le A^k (G)$.  Note that this implies that $\irr {G/A^k (G)} \subseteq \irr {G/K}$.  We have seen that for every character $\phi \in \irr K \setminus \{ 1_K \}$ that $\phi$ does not have an extension in $\irr {G/A^k (G)}$.  On the other hand, we do know by Lemma \ref{cent bij} does have an extension $\tilde\phi \in \irr {G/G'}$.  We claim that $\irri {k,p'}{G \mid \phi} \cap \irr {G/A^k (G)}$ is empty.  Suppose $\gamma \in \irri {k,p'}{G \mid \phi} \cap \irr {G/A^k (G)}$.  We know from Lemma \ref{cent bij} that multiplication by $\tilde\phi$ is a bijection from $\irr {G/K}$ to $\irr {G \mid \phi}$, so there exists a character $\lambda \in \irr {G/K}$ so that $\lambda \tilde\phi = \gamma$.  Since $\tilde\phi$ and $\gamma$ are linear, it must be that $\lambda$ is linear.  It follows that $\lambda \in \irr {G/G'K}$.  Notice that $G'K \le A^k (G)$, so $\irr {G/A^k (G)}$ is a subgroup of $\irr {G/G'K}$; so $\gamma$ and $\lambda$ both lie in $\irr {G/G'K}$.  This implies that $\tilde\phi \in \irr {G/G'K} \subseteq \irr {G/K}$ which is a contradiction.  This proves the claim.

We have noted that $\irri {k,p'}G \cap \irr {G/G'} = \irr {G/A^k (G)}$.  The previous paragraph shows that $\irr {G/A^k (G)} \subseteq \irri {k,p'}{G/K}$.  Thus, the characters in $\irri {k,p'}G \setminus \irri {k,p'}{G/K}$ are all nonlinear and thus have degrees that are at least $2$.  We can now apply Lemma \ref{averages} to see that $\acdd {k,p'} {G/K} \le \acdd {k,p'}G$.
\end{proof}

\section{Semi-direct products}

In this section, we will consider groups of the form $VH$ where $V$ is an irreducible, faithful module of the group $H$ and $V$ has characteristic $p$ for some prime $p$.  We break our work up into two case depending on whether or not $H$ is abelian.  We will see that this is the base case for our induction.  Note that throughout this section, the hypothesis that the field contain $p$th roots of unity is necessary.  To see this, observe that if $|G|$ is odd, then $1_G$ is the only rational valued irreducible character, and so, $\acdd {\QQ}G = 1$.  Thus, none of the results in this section would be true for odd order groups if the field did not contain the $p$th root unities.

In this next lemma, we see that $|H|$ must divide $|N| - 1$ which implies that $|H|$ is $p'$.  By It\^o's theorem (Theorem 6.15 of \cite{text}), this implies that every character in $\irr G$ has $p'$-degree.  In particular, $\irri {k,p'}G = \irri kG$, and hence, $\acdd {k,p'}G = \acdd kG$.  Also, note that the only values of $p^a$ that can occur in conclusion (2) are $4$ and $7$ since $|H|$ must divide $p^a - 1$.

\begin{lemma} \label{abelian 3}
Assume $H$ acts faithfully on an irreducible module $V$ of order $p^a$, and suppose that $k$ is an extension of $Q$ that contains the primitive $p$th roots of unity.  Let $G = HV$.  Suppose $H$ is a nontrivial abelian group, then $\acdd kG \ge 2$ except in the following cases:
\begin{enumerate}
\item if $|H| = 2$, then $\acdd kG = 2(p^a + 1)/(p^a + 3)$,
\item if $|H| = 3$, $A^k (H) = 1$, and $p^a < 10$, then $\acdd kG = 3(p^a + 2)/(p^a +8)$,
\item if $|H| = p^a - 1$ and $A^k (H) = 1$, then $\acd G = 2(p^a - 1)/p^a$.
\end{enumerate}
\end{lemma}

\begin{proof}
Since $V$ is irreducible and faithful, it follows that $H$ is cyclic (see Lemma 0.5 of \cite{MaWo}).  Because $H$ is cyclic, if $\lambda \in \irr V \setminus \{1_V \}$, then $H_{\lambda} = 1$.  This implies that every orbit of $H$ on $\irr V \setminus \{ 1_V \}$ has length $|H|$.  Since $V$ is an elementary abelian $p$-group, the values of $\lambda$ are $k$th roots of unity, so the values of $\lambda$ lie in $k$, and hence, $\lambda^G \in \irri kG$.  We obtain
$$
|\irri kG| = |\irr {H/A^k (H)}| + \frac {|V| - 1}{|H|} = |H:A^k (H)| + \frac {p^a -1}{|H|}.
$$
Getting a common denominator, we obtain
$$
|\irri kG| = \frac {|H||H:A^k (H)| + p^a - 1}{|H|},
$$
and
$$
\sum_{\chi \in \irri kG} \chi (1) = \sum_{\chi \in \irr {H/A^k (H)}} \chi (1) + |H| \frac {|N| - 1}{|H|} = |H:A^k (H)| + p^a - 1.
$$
We deduce that
$$
\acdd kG = \frac {|H:A^k (H)| + p^a - 1}{\frac {|H||H:A^k (H)| + p^a -1}{|H|}} = \frac {|H| (|H:A^k (H)| + p^a - 1)}{|H||H:A^k (H)| +p^a - 1}.
$$

Note that if $|H| = 2$, then $A^k (H) = 1$; so this equation becomes $\acdd kG = 2 (2 + p^a - 1)/(2 \cdot 2 + p^a - 1) = 2(p^a + 1)/(p^a + 3)$.  This yields Exception 1.  Next, if $|H| = 3$ and $A^k (H) = 1$, then we have $\acdd kG = 3(3 + p^a - 1)/(3 \cdot 3 + p^a - 1) = 3(p^a + 2)/(p^a + 8)$.  If $p^a \ge 10$, then $3 (p^a +2) = 3p^a + 6 \ge 2 p^a + 16 = 2(p^a +8)$, and so, $3(p^a + 2)/(p^a + 8) \ge 2$.  Note that we have Exception 2 if $p^a < 10$.  Suppose that $|H| = 3$ and $A^k (H) = H$.  This implies that $\acdd kG = 3 (1 + p^a - 1)/(3 \cdot 1 + p^a - 1) = 3p^a/(p^a + 2)$.  Notice that $p^a \ge |H| + 1 = 4$, so $3p^a \ge 2p^a + 4 = 2 (p^a + 2)$.  We deduce that $3p^a/(p^a + 2) \ge 2$ as desired.

For the remainder of the proof, we assume that $|H| \ge 4$.  Next, we suppose that $|H| = p^a - 1$ and $A^k (H) = 1$.  Our equation now becomes $\acdd kG = (p^a - 1) ((p^a - 1) + p^a - 1)/((p^a - 1)(p^a - 1) + p^a - 1) = 2(p^a - 1)^2/(p^a-1)(p^a - 1 + 1) = 2(p^a - 1)/p^a$.  This yields Exception 3.

We now can assume that either $|H| < p^a - 1$ or $A^k (H) > 1$.  In either case,$|H:A^k (H)| < p^a - 1$.  We observe that $|H|(|H:A^k (H)| + p^a - 1) = |H||H:A^k (H)| + (|H| - 2)(p^a-1) + 2 (p^a - 1)$.  Since $|H| \ge 4$, we have that $2|H| \ge |H| + 4$, and so, $|H| \ge |H|/2 + 2$.  This implies that $|H| - 2 \ge |H|/2$.  Notice that $|H:A^k (H)|$ divides $|H|$ which divides $p^a - 1$; so $|H:A^k (H)| < p^a - 1$ implies that $|H:A^k (H)| \le (p^a - 1)/2$.  We compute
$$
(|H| - 2) (p^a - 1) \ge (|H|/2) (p^a - 1) = |H| (p^a - 1)/2 \ge |H||H:A^k (H)|.
$$
We conclude that
$$
|H|(|\frac H{A^k (H)}| + p^a - 1) \ge |H||\frac H{A^k (H)}| + |H||\frac H{A^k (H)}| + 2 (p^a - 1).
$$
This yields the inequality
$$
|H|(|\frac H{A^k (H)}| + p^a - 1) \ge 2 (|H||\frac H{A^k (H)}| + p^a - 1),
$$
and therefore, $\acdd kG \ge 2$.
\end{proof}

Combining Lemma \ref{calculus} with Lemma \ref{abelian 3}, we obtain the needed lower bound on $\acd kG$.

\begin{corollary} \label{abelian 4}
Assume $H$ acts faithfully on an irreducible module $V$ of order $p^a$, and suppose that $k$ is an extension of $Q$ that contains the primitive $p$th roots of unity.  Let $G = HV$.  If $H$ is a nontrivial abelian group, then $\acdd kG \ge (2p + 2)/(p + 3)$.
\end{corollary}

\begin{proof}
Observe that $p^a \ge 3$, so $p^a (p^a + 1) = p^{2a} + p^a \le p^{2a} + 2p^a - 3 = (p^a - 1)(p^a + 3)$.  This implies that $2 (p^a + 1)/(p^a +3) \le 2 (p^a - 1)/p^a$.  Also, we have that $3(p^a + 2)(p^a +3) = 3p^{2a} + 15 p^a + 18 = 2p^{2a} + (p^a + 15) p^a + 18 \ge 2 p^{2a} + 18 p^a + 16 = 2 (p^a + 1) (p^a + 8)$, and this implies that $2 (p^a + 1)/(p^a +3) \le 3 (p^a + 2)/(p^a + 8)$.  Finally, $2 (p^a +1)/(p^a + 3) < 2$.  In light of Theorem \ref{abelian 3}, we have $2(p^a+1)/(p^a+3) \le \acdd kG$ in all cases.  Using Lemma \ref{calculus}, we see that $2(p+1)/(p+3) \le 2(p^a+1)/(p^a+3)$.  Combining, we obtain the desired conclusion.
\end{proof}

We now turn to the case where $H$ is nonabelian.  The proof of the following theorem should be compared to the proof of Lemma 9.2 of \cite{nguyen}, and our proof is motivated by the proof there.

\begin{theorem} \label{nonabelian 3}
Assume $H$ acts faithfully on an irreducible module $V$ of characteristic $p$, and suppose that $k$ is an extension of $Q$ that contains the primitive $p$th roots of unity.  Let $G = HV$.  If $H$ is a nonabelian group, then $\acdd {k,p'}G \ge 2$.
\end{theorem}

\begin{proof}
We first claim that either $p = 2$ and $G \cong S_4$ or $H$ has an orbit on $\irr V \setminus \{ 1_V \}$ whose length is not divisible by $p$ and is at least $4$.

Observe that $\irr V$ can also be viewed as a faithful irreducible module for $H$ of order $p^a$.  Suppose $\alpha \in \irr V \setminus \{ 1_V \}$, and let $T$ be the stabilizer of $\alpha$ in $H$.  We see that $T$ must be core-free and thus, $H$ is isomorphic to a subgroup of $S_l$ where $l = |H:T|$.  Since $H$ is nonabelian, we cannot have $l = 2$.  Thus, we see that every orbit of $H$ on $\irr V \setminus \{ 1_V \}$ has length at least $3$.

If $l = 3$, then $H$ is isomorphic to a subgroup of $S_3$, and since $H$ is nonabelian, we must have $H \cong S_3$.  Since $V$ is irreducible, this implies that $|V| = p^2$.  If $p = 2$, then $G \cong S_4$.  Thus, we may assume that $p$ is odd.  Since $H$ acts faithfully on $V$, we must have $p \ge 5$.  Notice that every character in $\irr V \setminus \{ 1 \}$ that lies in an $H$-orbit of size $3$ must be stabilized by a Sylow $2$-subgroup of $H$.  It is not difficult to see that the centralizer in $\irr V$ of each Sylow $2$-subgroup of will have size at most $p$.  Since $H$ contains three Sylow $2$-subgroups, this implies that $\irr V \setminus \{ 1_V \}$ contains at most $3(p-1)$ irreducible characters in $H$-orbits of size $3$.  Since $|\irr {V} \setminus 1| = p^2 - 1 = (p+1)(p-1) \ge 6(p-1) > 3(p-1)$.  It follows that $\irr V \setminus \{ 1_V \}$ lies in an $H$-orbit of length greater than $3$.  Observe that the length of this orbit must be $6$, and since $p \ge 5$, the length of this orbit is not divisible by $p$.

We may now assume that every $H$-orbit on $\irr V \setminus \{ 1_V \}$ has length at least $4$.  We know that the lengths of the $H$-orbits of $\irr V \setminus \{ 1_V \}$ must sum to $p^a - 1$, so at least one of the orbits must have length that is not divisible by $p$.  This proves the claim.

Since the irreducible characters of $S_4$ are rational valued,
$$
\acdd {k,2'}{S_4} = (1 + 1 + 2 + 3 + 3)/5 = 10/5 = 2,
$$
and so, we have the desired conclusion when $p = 2$ and $G \cong S_4$.  We may assume that $H$ has an orbit on $\irr V \setminus \{ 1_V \}$ whose length is not divisible by $p$ and is at least $4$.  Let $\alpha_1, \dots, \alpha_l$ be representatives for the $H$-orbits of $\irr V \setminus \{ 1_V \}$ whose lengths are not divisible by $p$, and (renumbering if necessary) assume that $\alpha_1$ is an orbit of length at least $4$.  Let $T_i$ be the stabilizer in $H$ of $\alpha_i$.  Observe that $VT_i$ will be the stabilizer of $\alpha_i$ in $G$.  By the Fundamental Counting Principle, we know that the $|H:T_i|$ will equal the length of the $H$-orbit of $\alpha_i$.

Note that since $T_i$ stabilizes $\alpha_i$ and $\alpha_i$ is linear, $T_i$ will centralize $H/\ker {\alpha_i}$.  It follows that $VT_i/\ker{\alpha_i} \cong V/\ker{\alpha_i} \times T_i$, and $\irr {VT_i \mid \alpha_i} = \{ \alpha_i \times \tau \mid \tau \in \irr {T_i} \}$.  Notice that the values that $\alpha_i$ takes on are $p$th roots of unity; so the values $\alpha_i$ take on lie in $k$.  If $\gamma \in \irr {T_i}$, then $\alpha_i \times \gamma \in \irri {k,p'}{VT_i \mid \alpha_i}$ if and only if $\gamma \in \irri {k,p'}{T_i}$.  Furthermore, if $\gamma \in \irri {k,p'}{T_i}$, then $(\alpha_i \times \gamma)^G \in \irri {k,p'}G$ by Clifford's theorem (Corollary 6.11 of \cite{text}) and the fact that $p$ does not divide $|H:T_i|$.  Suppose $\gamma \not\in \irri k{T_i}$.  This implies that $k[\gamma]$ is a proper extension of $k$, and so, there will be a nontrivial Galois automorphism $\sigma$ of $k[\gamma]$ that fixes $k$ and does not fix $\gamma$.  Then $\gamma^\sigma \ne = \gamma$, and so, $\alpha_i \times \gamma^\sigma \ne \alpha_i \times \gamma$.  By the Clifford correspondence (Corollary 6.11 of \cite{text}), we have that $(\alpha_i \times \gamma^\sigma)^G \ne (\alpha_i \times \gamma^\sigma)^G$.  It is not difficult to see that $\alpha_i^\sigma = \alpha_i$, so $((\alpha_i \times \gamma)^G)^\sigma = ((\alpha_i \times \gamma)^\sigma)^G = (\alpha_i \times \gamma^\sigma)^G \ne (\alpha_i \times \gamma)^G$.  It follows that the values of $(\alpha_i \times \gamma)^G$ do not lie in $k$.  This implies that $|\irri {k,p'}{G \mid \alpha_i}| = |\irri {k,p'}{VT_i \mid \alpha_i}| = |\irr {k,p'}{T_i}|$.

It follows that
$$
|\irri {k,p'}G| = |\irri {k,p'}H| + \sum_{i=1}^l |\irri {k,p'}{T_i}|.
$$
Using Corollary 6.11 of \cite{text}, we have $\irri {k,p'}{G \mid \alpha_i} = \{ (\alpha_i \times \tau)^G \mid \tau \in \irri {k,p'}{T_i} \}$.  Since $|G:VT_i| = |H:T_i|$, it follows that the degrees of the characters in $\irri {k,p'}{G \mid \alpha_i}$ are $\{ |H:T_i| \tau (1) \mid \tau \in \irri {k,p'}{T_i} \}$.  With this in mind, we obtain
$$
\sum_{\chi \in \irri {k,p'}G} \chi (1) = \!\!\!\!\! \sum_{\gamma \in \irri {k,p'}H} \!\!\!\!\! \gamma (1) + \sum_{i=1}^l |H:T_i| \left(  \sum_{\theta \in \irri {k,p'}{T_i}} \!\!\!\!\! \theta (1) \right)
$$
Observe that $\irri {k,p'}H$ will be the union of the linear characters of $H$ having values in $k$ with the nonlinear irreducible characters of $H$ having $p'$-degree and values in $k$.  We write $\nli {k,p'}H$ for the set of nonlinear irreducible characters of $H$ with degrees not divisible by $p$ an having values in $k$.  We obtain $|\irri {k,p'}H| = |H:A^k (H)| + |\nli {k,p'}H|$ and
$$
\sum_{\gamma \in \irri {k,p'}H} \gamma (1) \ge |H:A^k (H)| + 2 |\nli {k,p'}H|.
$$
Also, observe that
$$
\sum_{\theta \in \irri {k,p'}{T_i}} \!\!\!\!\! \theta (1) \ge |\irri {k,p'}{T_i}|.
$$
This yields
$$
\sum_{\chi \in \irri {k,p'}G} \chi (1) \ge |H:A^k (H)| + 2 |\nli {k,p'}H| + \sum_{i=1}^l |H:T_i| |\irri {k,p'}{T_i}|.
$$

Remember that $\alpha_1$ was chosen so that $|H:T_1| \ge 4$.  As we mentioned earlier, $A^k (T_1) \le A^k (H)$, and since $H$ is nonabelian and $T_1$ is core-free, we see that $A^k (T_1) < A^k (H)$.  (If $A^k (H) = A^k (T)$, then $1 < A^k (H) \le T$ and this contradicts the fact that $T_1$ is core-free.)  This implies that $|H:A^k (t)| \ge 2|H:A^k (H)|$.  Also, we know that every linear character of $T_1$ with values in $k$ is contained inside $\irri {k,p'}{T_1}$, so we have $|T:A^k (T_1)| \le |\irri {k,p'}{T_1}|$.  Since $|H:T_1| \ge 4$, we see that $2|H:T_1| - 4 \ge |H:T_1|$ and hence, $|H:T_1| - 2 \ge |H:T_1|/2$.  Combining these, we have
$$
(|\frac H{T_1}|-2)|\irri {k,p'}{T_1}| \ge \frac {|\frac H{T_1}|}2 \frac{|\frac {T_1}{A^k ({T_1})}|}1 = \frac {|\frac H{A^k ({T_1})}|}2 \ge |\frac H{A^k (H)}|.
$$
This implies that
$$
|\frac H{T_1}| |\irri {k,p'}{T_1}| = (|\frac H{T_1}| - 2) |\irri {k,p'}{T_1}| + 2 |\irri {k,p'}{T_1}|.
$$
We obtain the inequality
$$
|\frac H{T_1}| |\irri {k,p'}{T_1}| \ge |\frac H{A^k (H)}| + 2|\irri {k,p'}{T_1}|,
$$
and therefore, as $|H:T_i| \ge 2$ for all $2 \le i \le k$, we calculate that
$$
|\frac H{A^k (H)}| + 2 |\nli {k,p'}H| + \sum_{i=1}^l |\frac H{T_i}| |\irri {k,p'}{T_i}| \ge
$$
$$
2 (|\frac H{A^k (H)}| + |\nli {k,p'}H| + \sum_{i=1}^l |\irri {k,p'}{T_i}|).
$$
We conclude that $\acdd {k,p'}H \ge 2$ as desired.
\end{proof}

We can obtain $\acdd kG$ in this same situation.

\begin{corollary}\label{nonabelian 4}
Assume $H$ acts faithfully on an irreducible module $V$ of characteristic $p$, and suppose that $k$ is an extension of $Q$ that contains the primitive $p$th roots of unity.  Let $G = HV$.  If $H$ is a nonabelian group, then $\acdd kG \ge 2$.
\end{corollary}

\begin{proof}
By Theorem \ref{nonabelian 3}, we have that $\acdd {k,p'}G \ge 2$.  Notice that all the characters used to compute $\acdd kG$ that were not used for $\acdd {k,p'}G$ have degrees at least $2$, so using Lemma \ref{averages} we have $\acdd kG \ge 2$.
\end{proof}

\section{Main Theorem}

We are now ready to prove the theorems in the Introduction.  The results are all corollaries of this next theorem.  The proof of this theorem should be compared to the proof of Theorem 9.3 of \cite{nguyen}.  The astute reader will that our arguments are closely related to the arguments there.  Again, we note that the hypothesis that $k$ contain the $p$th roots of unity is necessary when $p$ is odd for this theorem since if $G$ has odd order, then $\acdd {\QQ}G = 1$, and when $p$ is an odd prime there exist many groups of odd order that do not have a normal $p$-complement.

\begin{theorem}\label{main}
Let $p$ be a prime, let $k$ be an extension of $Q$ that contains the primitive $p$th roots of unity, and let $G$ be a solvable group.  Assume one of the following:
\begin{enumerate}
\item $p$ is odd and $\acdd {k,p'}G < 2(p+1)/(p+3)$,
\item $p = 2$, $k = Q$, and $\acdd {k,2'}G < 2$,
\item $p = 7$, $|G|$ is odd, $k$ contains a cube root of unity, and $\acdd {k,7'}G < 9/5$,
\item $p = 7$, $|G|$ is odd, $k$ does not contain a cube root of unity, and $\acdd {k,7'}G < 2$,
\item $p$ is an odd prime other than $7$, $|G|$ is odd, and $\acdd {k,p'}G < 2$.
\end{enumerate}
Then $G$ has a normal $p$-complement.
\end{theorem}

\begin{proof}
We work by induction on $|G|$.  If $G$ is abelian, then the conclusion is immediate.  Thus, we assume that $G$ is not abelian.  This implies that $G' > 1$, and so, $G'$ contains a minimal normal subgroup $N$ of $G$.  Since every character in $\irri {k,p'}G$ that is not in $\irri {k,p'}{G/N}$ is nonlinear, they all have degrees at least $2$, and since $\acdd {k,p'}G < 2$, we conclude via Lemma \ref{averages} that $\acdd {k,p'}{G/N} \le \acdd {k,p'}G$, so $G/N$ will satisfy the inductive hypothesis, so $G/N$ has a normal $p$-complement $K/N$.  If $N$ is a $p'$-group, then $K$ will be a normal $p$-complement of $G$, and the result is proved.

We now suppose that $N$ is a $p$-group.  Let $K_1$ be a Hall $p$-complement of $K$ and observe that $K$ is a Hall $p$-complement of $G$.  By the Frattini argument, we have $G = K N_G (K_1) = N K_1 N_G (K_1) = N N_G (K_1)$.  Let $H = N_G (K_1)$.  If $G = H$, then $K_1$ will be the normal $p$-complement for $G$, and we are done.  Thus, we may assume that $H < G$.  Notice that $H \cap N$ is normal in $H$ since $N$ is normal in $G$ and $H \cap N$ is normal in $N$ since $N$ is abelian.  Thus, $H \cap N$ is normal in $G$ and proper in $N$.  By the minimality of $N$, we see that $H \cap N = 1$.

Suppose $H$ contains a subgroup $M$ so that $M$ is minimal normal in $G$.  If $M \le G'$, then every character in $\irri {k,p'}G$ that is not in $\irri {k,p'}{G/M}$ has degree at least $2$ and $\acdd {k,p'}G < 2$; so using Lemma \ref{averages}, we see that $\acdd {k,p'}{G/M} \le \acdd {p'}G$.  On the other hand, if $M \cap G' = 1$, then by Theorem \ref{acd cent k} we have $\acdd {k,p'}{G/M} = \acdd {k,p'}{G}$.  In both cases, we see that $G/M$ satisfies the inductive hypothesis, and so, $G/M$ has a normal $p$-complement.  This implies that $K_1 M$ is normal in $G$.  Since $K_1M \le H$, we have $N \cap K_1 M = 1$, and so, $N$ centralizes $K_1 M$.  In particular, $N$ centralizes $K_1$ which is a contradiction since $N$ does not normalize $K_1$.  Therefore, $H$ must be core-free.

We now see that $N$ can be viewed as an irreducible, faithful module for $H$.  If $H$ is abelian, then note that $H$ must be cyclic and so every $H$-orbit on $N \setminus \{ 1 \}$ will have length $|H|$.  This implies that $|H|$ divides $|N| - 1$.  By It\^o's theorem (Theorem 6.15 of \cite{text}), every character in $\irri kG$ has $p'$-degree, and so, $\acdd {k,p'}G = \acdd kG$.  Applying Corollary \ref{abelian 4}, we know that $\acdd {k,p'}G = \acdd kG \ge 2(p+1)/(p+3)$, and this is a contradiction if we have Hypothesis 1.  If we have Hypothesis 2, then $|H|$ must be odd and since $k = Q$, this implies that $A^k (H) = H$.  Applying Lemma \ref{abelian 3}, we have $\acdd {Q,2'} (G) \ge 2$, a contradiction.  In Hypothesis 3, since $|G|$ is odd, we see that $|H|$ is odd.  If $|H| = 3$, then since $k$ contains primitive $3$rd roots of unity, we see that $A^k (H) = 1$, and so Lemma \ref{abelian 3} (2) yields $\acdd kG = 9/5$.  Notice that $7^a - 1$ must be even, so none of the other exceptions to Lemma \ref{abelian 3} can occur, and $\acdd kG \ge 2$, and so we have a contradiction.  In Hypotheses 4 and 5, we see that $|H|$ is odd, $p^a - 1$ is even and if $|H| = 3$ then $p^a \ge 10$ except when $p^a = 7$ and since $k$ does not contain primitive $3$rd roots of unity, we have $A^k (H) = H$.  In particular, none of the exceptions to Lemma \ref{abelian 3} occur, and we have $\acdd kG \ge 2$ which is a contradiction.  If $H$ is nonabelian, then $\acdd {k,p'}G \ge 2$ by Theorem \ref{nonabelian 3}.  In all cases, we have a contradiction, and hence, the theorem is proved.
\end{proof}

Theorem \ref{first} follows from Theorem \ref{main} (1) using $k = \CC$.  Obviously, Theorem \ref{second} (1) is Theorem \ref{main} (2).  The characters in $\irri {\QQ}G$ that do not lie in $\irri {\QQ,2'}G$ are all nonlinear, so they have degrees at least $2$, so Theorem \ref{second} follows from Theorem \ref{second} (1) combined with Lemma \ref{averages}.  Similarly, since every real valued linear character is rational valued, the characters in $\irri {\RR,2'}G$ and $\irri {\RR}G$ that do not lie in $\irri {\QQ}G$ have degrees at least $2$.  Thus, Theorem \ref{second} (3) and (4) are consequences of Theorem \ref{second} (1) and Lemma \ref{averages}.  Taking $k = \QQ_p$, we obtain Theorem \ref{third} (1) using Theorem \ref{main} (1), and since the characters in $\irr {\QQ_p}G$ that are not in $\irr {\QQ_p,p'}G$ are nonlinear, Theorem \ref{third} (2) is an application of Theorem \ref{third} (1) and Lemma \ref{averages}.  Finally, Theorem \ref{fourth} (1) with $k = \CC$ is Theorem \ref{main} (3), Theorem \ref{fourth} (2) with $k = \CC$ is Theorem \ref{main} (5), and Theorem \ref{fourth} (3) is Theorem \ref{main} (4) with $k = \QQ_p$.

\end{document}